\documentclass[a4paper,11pt]{article}
\usepackage[utf8]{inputenc}
\usepackage[english]{babel}
\usepackage{amssymb}
\usepackage{hyperref,prettyref}
\usepackage{latexsym}
\usepackage{color}
\usepackage{graphicx}
\usepackage{amsfonts}
\usepackage{amsthm}
\usepackage{mathrsfs}
\usepackage{amsmath}
\usepackage{paralist}
\usepackage{graphics,cases}
\usepackage{epsfig}
\usepackage{graphicx}
\usepackage{epstopdf}
\newtheorem{theorem}{Theorem}
\newtheorem{lemma}{Lemma}
\newtheorem{corollary}{Corollary}

\newtheorem{proposition}{Proposition}

\everymath{\displaystyle}
\newcommand{\R}{\mathbb{R}}

\newcommand{\N}{\mathbb{N}}

\renewcommand\({\left(}
\renewcommand\){\right)}
\renewcommand\[{\left[}
\renewcommand\]{\right]}

\newcommand\la{\lambda}

\newcommand{\be}{\begin{equation}}
\newcommand{\ee}{\end{equation}}
\newcommand{\ba}{\begin{array}}
\newcommand{\ea}{\end{array}}
\newcommand{\bea}{\begin{eqnarray*}}
\newcommand{\eea}{\end{eqnarray*}}
\newcommand{\bean}{\begin{eqnarray}}
\newcommand{\eean}{\end{eqnarray}}

\begin{document}

\makeatletter \@addtoreset{equation}{section}
\renewcommand{\theequation}{\thesection.\arabic{equation}}
\makeatother


\renewcommand{\baselinestretch}{1}
\renewcommand{\theequation}{\arabic{section}.\arabic{equation}}
\renewcommand{\thetheorem}{\arabic{section}.\arabic{theorem}}
\renewcommand{\thelemma}{\arabic{section}.\arabic{lemma}}
\renewcommand{\theproposition}{\arabic{section}.\arabic{proposition}}
\renewcommand{\thedefinition}{\arabic{section}.\arabic{definition}}
\renewcommand{\thecorollary}{\arabic{section}.\arabic{corollary}}
\renewcommand{\theremark}{\arabic{section}.\arabic{remark}}
\pagestyle{empty}

\pagestyle{myheadings}

\title{Positive and negative exact boundary controllability results for the linear Biharmonic Schr\"odinger equation}

\author{Ka\"{\i}s Ammari
\thanks{LR Analysis and Control of PDEs,  LR22ES03, Department of Mathematics, Faculty of Sciences of
Monastir, University of Monastir, 5019 Monastir, Tunisia,
\url{kais.ammari@fsm.rnu.tn}.}~~~~~~
Hedi Bouzidi\thanks{LR Analysis and Control of PDEs,  LR22ES03, Department of Mathematics, Faculty of Sciences of
Monastir, University of Monastir, 5019 Monastir, Tunisia,
\url{hedi.bouzidi@fst.utm.tn}.}}

\date{}

 \maketitle

\begin{abstract} In this paper, we study the exact boundary controllability of the linear Biharmonic Schr\"odinger equation $i\partial_ty=-\partial_x^4y+
\gamma\partial_x^2y$ on a bounded domain with hinged boundary conditions and boundary control acts on the second spatial derivative at the {left} endpoint, where the parameter $\gamma<0$. We prove that this system is exactly controllable in time $T>0$, if and only if, the parameter $\gamma$  does not belong to a critical countable set of negative real numbers.  The analysis in this work is based on spectral analysis together with the nonharmonic Fourier series method.
\end{abstract}
~~~~~~~~~~\\
~~~~~~~~~~\\
{\it{2010 Mathematics Subject Classification}}. 35P05, 35G05, 81Q10, 81Q93, 93C15, 93D15. \\
{\it{Key words and phrases}}. Optical fibers, laser beams, Biharmonic Schr\"odinger, boundary control.
\maketitle
\section{Introduction and main result}
In this paper, we study the boundary controllability of
the linear fourth-order Schr\"odinger equation on the bounded interval $\(0,\ell\)$, where $\ell>0$. More precisely, we consider the following control system
\be \label{bihar2} \begin{cases} i\partial_ty(t,x)=-\partial_x^4y(t,x)+
\gamma\partial_x^2y(t,x), &(t,x)\in(0,T)\times(0,\ell),\\
y(t, 0)=y(t, \ell) = \partial_x^2y(t,\ell) =0,~\partial_x^2y(t,0) = f(t),&t\in(0,T),\\
y(0, x) = y_0(x),& x\in(0,\ell),
\end{cases}\ee
where the parameter $\gamma<0$, $f$ is a control that acts on the left end $x = 0$, and
the function $y_0$ is the initial condition. For system \eqref{bihar2}, the appropriate control notion to study is the
exact controllability, which is defined as follows: System \eqref{bihar2} is said to be exactly controllable
at time $T>0$ if, given any initial state $y_0$, there exists a control $f$ such that the corresponding
solution $y = y(t,x)$  satisfies $y(T,.) = 0$.\\
\indent The fourth-order Schr\"odinger equation has been modeled by Karpman \cite{Karpman}, and Karpman $\&$ Shagalov \cite{Karpman1} in order to describe the propagation of intense laser beams in a bulk medium with Kerr nonlinearity when small fourth-order dispersion is taken into an account. Nowadays, equation \eqref{bihar2} is often called the biharmonic Schr\"odinger equation
and {has various applications in several areas of physics, such as}  nonlinear optics, plasma physics and nonuniform dielectric media, see \cite{Ben, CuiGuovcb, Karlsson, Fibich}. \\
\indent Notice that, in the case where $\partial_x^4y\equiv0$, Equation \eqref{bihar2} collapses to the standard second order Schr\"odinger equation. In this direction, the exact controllability of the second order Schr\"odinger equation  has been extensively investigated, see for instance \cite{Ammari20170,Ammari20121bcb, Ammari2019}. In the case where $\gamma=0$, Zheng and Zhongcheng \cite{Zheng} proved that the biharmonic Schr\"odinger equation \eqref{bihar2} with clamped boundary conditions and Neumann boundary control is exactly controllable at time $T > 0$. Later on, Wen et al. \cite{Wen1} established well-posedness and exact controllability results for a system described by the fourth order Schr\"odinger equation in \eqref{bihar2} for $\gamma=0,$ on a bounded domain of $\R^n~(n\geq2)$ with boundary control and collocated observation. Along similar lines, the same authors in \cite{Wen2}, extended these results to the case of a linear fourth-order multi-dimensional Schr\"odinger equation with hinged boundary control and collocated observation. Recently, in the case where $\gamma=1,$  Capistrano-Filho and Cavalcante \cite{Capistrano} proved global stabilization and exact controllability results for the biharmonic Schr\"odinger equation \eqref{bihar2} on a periodic domain $\mathbb{T}$ with internal control supported on an arbitrary sub-domain of $\mathbb{T}$. Whereas, in the case where $\gamma(x)\geq0$, the authors \cite{Ammari20225} showed that the fourth-order Schr\"odinger equation \eqref{bihar2} with clamped boundary conditions and boundary control acting on the first spatial derivative at one endpoint is exactly controllable at time $T > 0$.\\
 \indent As we know, the exact controllability of System \eqref{bihar2} with $\gamma<0$ has not been studied in the literature yet. As we will see, that the parameter $\gamma$ plays a key role in understanding the dynamics of System \eqref{bihar2}. Indeed, we first prove that the
eigenvalues $\({\lambda_{n}}\)_{n\in\N^*}$ associated to
Problem \eqref{bihar2} with $f(t)\equiv0,$ are bounded from below by the quantity $-\frac{\gamma^2}{4}$. Furthermore, we show that the finite number of negative eigenvalues $\({\lambda_{n}}\)_{n\in\N^*}$ are algebraically simple, if and only if, the parameter $\gamma$  does not belong to the following critical countable set of negative real numbers:
\be {\label{set1}{\Gamma^*}=\left\{- \, \frac{\pi^2}{\ell^2}\(p^2+q^2\)~:~p,q\in\N^*,~1\leq p<q\right\}.}\ee
Therefore, using the nonharmonic Fourier series method, we establish  positive and negative exact controllability results for System \eqref{bihar2}. Namely, our main result is the following:
\begin{theorem}
\label{hj} Given $T>0$ and
 $y_0\in H^{-1}\(0,\ell\) $. Assume that $\gamma\not\in\Gamma^*$. Then, there
exists a control $f\in L^{2}(0,T)$ such that the solution $y$ of the
control problem \eqref{bihar2}
 satisfies
\begin{equation*}
y(T,x)=0,~~x\in\[0,\ell\].
\end{equation*}
 Moreover, if $\gamma\in\Gamma^*,$ then System \eqref{bihar2} is not exactly controllable at time $T>0$.
\end{theorem}
The remainder of this paper is organized as follows: In Section $2$, we investigate the main properties of all the
eigenvalues $\({\lambda_{n}}\)_{n\in\N^*}$ associated to
the control system \eqref{bihar2}. In Section $3$, we prove the observability estimate for uncontrolled system \eqref{bihar2} with $f(t)\equiv0$. Finally, in Section $4$, we give the proof of our main exact controllability result.
\section{Spectral analysis}\label{Spe}
We consider the following spectral problem which arises by applying
separation of variables to system \eqref{bihar2} with $f(t)\equiv0$,
\begin{align}\label{biharr1}\begin{cases}
               \phi''''-\gamma\phi''=\lambda \phi,~~x\in(0,\ell),  \\
               \phi(0)=\phi''(0)=\phi(\ell)=\phi''(\ell)=0.
             \end{cases}\end{align}
Let $L^2(0,\ell)$ be the Lebesgue space of all functions $y$ defined on $(0,\ell),$ being equipped with the norm $$ \|y\|_{L^2(0,\ell)}= \(\int_0^\ell |y(x)|^2dx\)^{\frac{1}{2}},~\forall~ y\in L^2(0,\ell).$$
We consider the Sobolev space $H^{2}(0,\ell)\cap H^{1}_0(0,\ell)$
endowed with the norm $$
\|y\|_{H^{2}\cap H^{1}_0(0,\ell)}=\|{y}''\|_{L^2(0,\ell)},~\forall~y\in H^{2}(0,\ell)\cap H^{1}_0(0,\ell).$$
We introduce the operator
$\mathcal{A}$ defined in $L^2(0,\ell)$ by setting:
 $$ \mathcal{A}y = \phi''''-\gamma \phi'', $$
 on the domain
 $$ \mathcal D\(\mathcal{A}\):= \left\{\phi\in H^{4}(0,\ell)~:~\phi(0)=\phi''(0)=\phi(\ell)=\phi''(\ell)=0\right\},$$
which is dense in $L^2(0,\ell)$.
\begin{proposition}\label{rr}
The linear operator $\mathcal{A}$ is positive and self-adjoint such
that $\mathcal{A}^{-1}$ is compact. Moreover, the spectrum of
$\mathcal{A}$ is discrete and consists of a
 sequence of real eigenvalues
$(\lambda_{n})_{n\in\mathbb{N}^{*}}$ tending to $+\infty$:
{$$ -\frac{\gamma^2}{4} \leq\lambda_{1}\leq\lambda_{2}\leq \ldots\leq\lambda_{n}\leq \ldots
\underset{n\rightarrow +\infty}{\to}+\infty.$$} The corresponding eigenfunctions $(\Phi_{n})_{n\in\N^*}$ can be chosen
to form an orthonormal basis of $L^2(0,\ell)$.
\end{proposition}
\begin{proof}
Let $y \in\mathcal D\(\mathcal{A}\)$, then { using an }integration by parts, we
have \bea \langle \mathcal{A}\phi,
\phi\rangle_{L^2(0,\ell)}&=&\int_{0}^{\ell}\(\phi'''' (x) -\gamma \phi'' (x)\){\phi}(x) dx \nonumber\\
&=&\int_{0}^{\ell}\(\phi''(x)\)^{2}+\gamma\(\phi'(x)\)^{2}dx.
\eea Since $\displaystyle \int_{0}^{\ell}\({\phi''(x)}+\frac{\gamma\phi(x)}{2}\)^{2}dx\geq0,$ then $$\langle \mathcal{A}\phi,
\phi\rangle_{L^2(0,\ell)}\geq -\frac{\gamma^2}{4}\|\phi\|^2_{L^2(0,\ell)}.$$
This implies that the quadratic
form is bounded from below by $-\frac{\gamma^2}{4}$, and then, the linear operator
$\mathcal{A}$ is symmetric. Furthermore, it is easy to show that $Ran(\mathcal{A}-iId)=L^2(0,\ell)$, and this means that
$\mathcal{A}$ is {self-adjoint}. Since, the space $H^{2}(0,\ell)\cap H^{1}_0(0,\ell)$ is continuously and compactly embedded in the space
$L^2(0,\ell)$, then $\mathcal{A}^{-1}$ is compact in
$L^2(0,\ell)$.
\end{proof}
\begin{theorem}\label{Lem2}
One has:\\
{\bf(a)} All the positives eigenvalues of Problem
\eqref{biharr1} are algebraically simple and satisfy \be \lambda_n=
\frac{n^2\pi^2}{\ell^2}\(\frac{n^2\pi^2}{\ell^2}+\gamma\) \hbox{ for }~ n\geq n_0:=\[\ell{\pi}^{-1}{\sqrt{|\gamma|}}
\]+1,\label{gb}\ee
where $\[\bullet\]$ denotes the integer part of $\bullet$.\\
{\bf(b)} $\lambda=0$ is a simple eigenvalue of Problem
\eqref{biharr1} if and only if \be \gamma=-\frac{p^2\pi^2}{\ell^2},~p\in\N^*.\label{guo}\ee
{\bf(c)} All the negatives eigenvalues of Problem
\eqref{biharr1} satisfy \be \lambda_n=
\frac{n^2\pi^2}{\ell^2}\(\frac{n^2\pi^2}{\ell^2}+\gamma\) \hbox{ for } n\leq\[\ell{\pi}^{-1}{\sqrt{|\gamma|}}
\]-1.\label{gh}\ee
Moreover, they are algebraically simple if and only if $\gamma\not\in \Gamma^*$.
\end{theorem}
\begin{proof}
{\bf (a)} Since $\lambda>0$, any non-trivial solution $\phi(x,\lambda)$ of the fourth-order linear differential equation \be \phi''''-\gamma\phi''-\lambda \phi=0,~~x\in(0,\ell),\label{gv}\ee
may be expressed as follows:
$$ \phi(x,\lambda)=C_1\sin(x\eta(\lambda))+C_2\cos(x\eta(\lambda))+C_3\sinh(x\bar\eta(\lambda))
+C_4\cosh(x\bar\eta(\lambda)),$$
for some constants $C_j,~j=1,2,3,4,$ where
$$\eta(\lambda)=\sqrt{\frac{\sqrt{\gamma^2+4\lambda}-\gamma}{2}} \hbox{ and } \bar\eta(\lambda)=\sqrt{\frac{\sqrt{\gamma^2+4\lambda}+\gamma}{2}}.$$
Using the fact that $\bar\eta(\lambda)>0,$ and the boundary conditions in \eqref{biharr1}, we deduce that the positive eigenvalues of \eqref{biharr1} are solutions of the characteristic equation $\sin(\ell\eta(\lambda))=0.$
Consequently, all positives eigenvalues of Problem
\eqref{biharr1} satisfy
$$\lambda_n=\frac{n^2\pi^2}{\ell^2}\(\frac{n^2\pi^2}{\ell^2}+\gamma\),
~n>\frac{\ell\sqrt{|\gamma|}}{\pi}, $$
and this proves relation \eqref{gb}. The corresponding
eigenfunctions $\phi(x,\lambda_n),$ up to a multiplicative constant, have the form  \be \phi(x,\lambda_n)=\sin(x\eta(\lambda_n)),~n>\frac{\ell\sqrt{|\gamma|}}{\pi},\label{hgm}\ee
and then, each eigenvalue $\lambda_n$ of Problem
\eqref{biharr1} is simple.\\
{\bf (b)} In this case, any non-trivial solution $\phi(x,\lambda)$  of Equation \eqref{gv} may be written as the following,
$$ \phi(x,\lambda)=C_1+C_2x+C_3\sin\(\sqrt{|\gamma|}x\)+C_4\cos\(\sqrt{|\gamma|}x\),$$
for some constants $C_j,~j=1,2,3,4.$ Using the boundary conditions at $x=0$, one gets $C_1=C_4=0$. If \eqref{guo} is not satisfied, then, from the boundary conditions at $x=\ell$, one has $C_2=C_3=0$. This implies that
$ \phi(x,\lambda)\equiv0,$ which is a contradiction. Conversely, if \eqref{guo} holds, then, $C_2=0$. Thus, the eigenfunction $\phi(x,\lambda)$ associated to the eigenvalue $\lambda=0$, up to a multiplicative constant, has the form \be \phi(x,\lambda)= \sin\(\sqrt{|\gamma|}x\).\label{dog1}\ee Therefore, $\lambda=0$ is a simple eigenvalue of Problem \eqref{biharr1}. \\
{\bf(c)} Since $\lambda<0$, any non-trivial solution $\phi(x,\lambda)$  of Equation \eqref{gv} may be expressed as follows:
$$ \phi(x,\lambda)=C_1\sin(x\xi(\lambda))+C_2\cos(x\xi(\lambda))+C_3\sin(x\bar\xi(\lambda))+
C_4\cos(x\bar\xi(\lambda)),$$
for some constants $C_j,~j=1,2,3,4,$ where
\be\xi(\lambda)=\sqrt{\frac{\sqrt{\gamma^2+4\lambda}-\gamma}{2}} \hbox{ and } \bar\xi(\lambda)=\sqrt{\frac{-\sqrt{\gamma^2+4\lambda}-\gamma}{2}}.\label{drog}\ee
Using the boundary conditions at $x=0$, one has \be \phi(x,\lambda)=C_1\sin(x\xi(\lambda))+C_3\sin(x\bar\xi(\lambda)).\label{rj}\ee Thus, by the boundary conditions at $x=\ell$, $\lambda$ is an eigenvalue of Problem \eqref{biharr1} if and only if the function \be \sin(\ell\xi(\lambda))\sin(\ell\bar\xi(\lambda))=0.\label{7}\ee
Therefore, relation \eqref{gh} is proved. Since, 
$$ \xi(\lambda) \not=\bar\xi(\lambda) \hbox{ and } \(\xi(\lambda)\)^2+\(\bar\xi(\lambda)\)^2=|\gamma|,$$
then, by relation \eqref{7}, the following cases must be examined :\\
{\bf (i)} $\dfrac{\xi(\lambda)}{\bar\xi(\lambda)}=\frac{p}{q}$, where  $p,q\in\N^*$ and $p\not=q$. Or equivalently, $\gamma\in \Gamma^*$. Then, by relation \eqref{rj} and the boundary conditions at $x=\ell$, the eigenfunctions $\phi(x,\lambda_n)$ of Problem \eqref{biharr1}, have the form
\be
  \phi(x,\lambda_n)= C_1\sin(x\xi(\lambda_n))+C_3\sin(x\bar\xi(\lambda_n)),\label{gdv}
\ee 
where $|C_1|+|C_3|\not=0.$
Thus, each eigenvalue $\lambda_n$ of Problem \eqref{biharr1} is of  multiplicity two.\\
{\bf (ii)} $\dfrac{\xi(\lambda)}{\bar\xi(\lambda)}\not=\frac{p}{q}$, where  $p,q\in\N^*$. Or equivalently,  $\gamma\not\in \Gamma^*$. Then, by relation \eqref{rj} and the boundary conditions at $x=\ell$, the eigenfunctions $\phi(x,\lambda_n)$ of Problem \eqref{biharr1}, up to a multiplicative constant, have the form  \be \phi(x,\lambda_n)=\sin(x\eta(\lambda_n)),~n<\frac{\ell\sqrt{|\gamma|}}{\pi}. \label{dog}\ee
Therefore, each eigenvalue $\lambda_n$ of Problem \eqref{biharr1} is a simple. The proof is then complete.
\end{proof}
As a consequence of Theorem \ref{Lem2}, we have the following result.
\begin{corollary}\label{SP2} One has:\\
{\bf (a)} Let $(\la_{n})_{n\geq n_0}$ be the increasing sequence of positive eigenvalues  of the
spectral problem \eqref{biharr1} given by relation \eqref{gb}. Then, the following uniform gap condition holds
 \be {\la_{n+1}}-{\la_n}>C,~\hbox{as } n\to\infty,\label{biharr4}\ee
for some positive constant $C>0$.\\
{\bf (b)} Let $(\Phi_{n})_{n\geq 1}$ be the sequence of normalized eigenfunctions of the spectral problem \eqref{biharr1} so that
$\lim_{n\to\infty}\|{\Phi}_n\|_{L^2_\rho(0,\ell)}=1$. Then, one has  \be \dfrac{\left|
{\Phi_{n}'(0)}\right|}{\sqrt[4]{\la_n}}= \sqrt{\frac{2}{\ell}},~\hbox{as } n\to\infty.\label{biharr18}\ee
\end{corollary}
Now, we give a characterization of some fractional powers of the linear
operator $\mathcal{A}$ which will be useful in the next section. According to Lemma \ref{rr}, the operator $\mathcal{A}$ is
positive and self-adjoint, and hence it generates a scale of
interpolation spaces $\mathcal{H}_{\theta}$, $\theta \in
\mathbb{R}$. For $\theta\geq0$, the space $\mathcal{H}_{\theta}$
coincides with $\mathcal D(\mathcal{A}^{\theta})$ and is equipped
with the norm $\|u\|_\theta^2=\langle \mathcal{A}^\theta u,
\mathcal{A}^\theta u\rangle_{L^2 (0,\ell)}$, and for $\theta< 0$ it
is defined as the completion of $L^2 (0,\ell)$ with respect to this
norm. Furthermore, we have the following spectral representation of
space $\mathcal{H}_{\theta}$,
\be \mathcal{H}_{\theta}=
\left\{u(x)=\sum\limits_{n\in\N^*}c_n\Phi_n(x)~:
 ~\|u\|_{\theta}^2=\sum\limits_{n\in\N^*}|\la_n|^{2\theta}|c_n|^{2}<\infty\right\},\label{bihar8}\ee
where $\theta\in \R$, and the eigenfunctions $\(\Phi_{n}\)_{n\in\N^*}$
are defined in Lemma \ref{rr}. In particular, \be \mathcal{H}_{0}=L^2 (0,\ell) \hbox{ and } \mathcal{H}_{1/4}=H^{1}_{0}(0,\ell).\label{ahme}\ee
\section{Observability of the biharmonic Schr\"odinger equation}
In this section, we prove some observability results for the uncontrolled system,
\be{ \label{biharr22}\begin{cases}
i\partial_tz(t,x)=-\partial_x^4z(t,x)+
\gamma\partial_x^2z(t,x), &(t,x)\in(0,T)\times(0,\ell),\\
z(t, 0) =\partial_x^2z(t,0) =z(t, \ell) =\partial_x^2z(t, \ell) = 0,&t\in(0,T),\\
z(0, x) = z_0,& x\in(0,\ell).
                \end{cases}} \ee
We start by mentioning the well-posedness of System \eqref{biharr22}. Obviously, the solutions of Problem \eqref{biharr22} can be written in Fourier series as \be
{z}(t,x):=\sum\limits_{n\in\N^*}c_ne^{i\la_nt}{\Phi}_n(x),\label{bihar9}\ee
where the Fourier coefficients are given by $$c_n:=\int_{0}^{\ell} z_0 (x) \bar{\Phi}_n(x)\rho(x)dx,~n\in\N^*,$$
and $\(c_n\)\in\ell^2\(\N^*\)$. Let us denote by $\mathcal{E}_{\theta}(t)$ the {energy} associated to the space $\mathcal{H}_{\theta}$, then \begin{align*}
                            \mathcal{E}_{\theta}(t) &= \|z\|_{\theta}^2=\sum\limits_{n\in \N^*}
|{\la}_n|^{2\theta}|c_{n}e^{i\la_nt}|^2\\&= \sum\limits_{n\in \N^*}
|{\la}_n|^{2\theta}|c_{n}|^2 =\mathcal{E}_{\theta}(0),
                          \end{align*}
which establishes the conservation of energy along time. Consequently, one has:
\begin{lemma}\label{pos}
Let $\theta\in\R$ and $z_0\in\mathcal{H}_{\theta}$. Then Problem \eqref{biharr22} has a
unique solution\\ ${z}\in C([0,T],\mathcal{H}_{\theta})$. Moreover, the energy $\mathcal{E}_{\theta}(t)$ of System \eqref{biharr22} is conserved along the time.
\end{lemma}
\begin{proposition}\label{ph1}
{ Assume that $\gamma\not\in\Gamma^*$, where $\Gamma^*$ is given by relation  \eqref{set1}.} Let $T
>0$ and $z_0\in H^1_0(0,\ell)$, then
\begin{equation}\label{biharr23}
C_T^{-1} \|z_0\|_{H^1_0(0,\ell)}^{2}\leq\int_{0}^{T}|\partial_xz(t,0)|^{2}dt\leq C_T \|z_0\|_{H^1_0(0,\ell)}^{2},
\end{equation}
for some positive constant $C_T>0$, depending on $T,$ where $z$ is the solution of the adjoint system \eqref{biharr22}.
\end{proposition}
\begin{proof}
 From \eqref{bihar9}, we have \be
\int_{0}^{T}|\partial_xz(t,0)|^{2}dt =
\int_{0}^{T}\Big{|}\sum\limits_{n\in \N^*}
c_{n}e^{i{\la}_n t} \Phi_n'\(0\)\Big{|}^{2}dt,~\forall T>0. \label{biharr24}\ee Since $\gamma\in\Gamma^*$, then by Theorem
\ref{Lem2} and the gap condition \eqref{biharr4}, Beurling's Lemma (e.g., \cite{Ammari20225})
states that, for any $T>2\pi D^+\(\la_n\)$, the family $\(e^{i\lambda_n
t}\)_{n\in\N^*}$ forms a Riesz basis of $L^2(0, T)$, where $ D^+\(\la_n\):= \displaystyle \lim_{r\to\infty}
\frac{n^+\(r,\lambda_n\)}{r}$ is the Beurling upper density of the sequence $(\lambda_{n})_{n\in\N^*},$ and $n^+\(r ,\lambda_n\)$  denotes the maximum number of terms of the sequence $(\lambda_{n})_{n\in\N^*}$ contained
in an interval of length $r$. Therefore, from relation \eqref{biharr24}, we deduce that
for every $T>2\pi D^+\(\la_n\),$ there exists a positive constant $C_T>0$, depending on $T$, such that \be C_T^{-1} \sum\limits_{n\in
\N^*}\left|c_n \Phi'_n(0)\right|^{2}\leq\int_{0}^{T}|\partial_xz(t,0)|^{2}dt\leq C_T \sum\limits_{n\in
\N^*}\left|c_n \Phi'_n(0)\right|^{2}.\label{biharr25}\ee
By relations  \eqref{gb} and \eqref{hgm}, we find that $ D^+\(\la_n\)=0.$ It is easy to check from relations \eqref{hgm}-\eqref{dog1} and \eqref{dog}, that the
eigenfunctions $\(\Phi_n\)_{n\in\N^*}$ satisfy \be
\Phi_n'(\ell)\not=0,~~\forall n\in\N^*. \label{biharr2}\ee
 Moreover, using the second statement of Corollary \ref{SP2}, one has  $$
C^{-1}\sqrt{\la_n}\leq\left|\Phi'_n(0)\right|^2\leq C\sqrt{\la_n},~\hbox{ as } n\to\infty,$$
for some positive constant $C> 0$.  Therefore, by relation \eqref{biharr25}, for any $T>0$, one gets
\be C_T^{-1} \sum\limits_{n\in
\N^*}\sqrt{\left|\lambda_n\right|}\left|c_n\right|^{2} \leq\int_{0}^{T}|\partial_xz(t,0)|^{2}dt\leq C_T \sum\limits_{n\in
\N^*}\sqrt{\left|\lambda_n\right|}\left|c_n\right|^{2},\label{gd}\ee
for some new constant $C_T>0$. Thus, by relations \eqref{bihar8} and \eqref{ahme}, we get \eqref{biharr23}.
\end{proof}
\section{Proof of Theorem \ref{hj}}
In this section, we prove our main controllability result, which is given by Theorem \ref{hj}.
Let us first mention that the control problem  \eqref{bihar2} is well
posed in the sense of "transposition", see \cite{J.L2}. More precisely, we have:
\begin{proposition} \label{b-p} Let
$T>0$, and  $f\in L^{2}(0,T)$. Then for any $y_0\in H^{-1}(0,\ell)$, there exists a unique weak solution  of
System \eqref{bihar2} in the sense of transposition, satisfying
\be y\in C\([0, T ]; H^{-1}(0,\ell)\).\label{biharr43}\ee
Moreover, there exists a constant $C(T)>0$ such that
\be \|y\|_{L^\infty\([0, T ]; H^{-1}(0,\ell)\)}\leq
C(T)\(\|y_0\|_{H^{-1}(0,\ell)}+\|f\|_{L^{2}(0,T)}\).\label{biharr44}\ee\end{proposition}
For the proof of this Proposition, the reader is referred to
\cite{Ammari20225,Zheng}. In that papers, the result is proven for the linear fourth-order Schr\"odinger equation with clamped boundary conditions and boundary control acts on the first spatial derivative at one endpoint, but the proof can easily be adapted to our setting.
We are now ready to prove Theorem \ref{hj}.
\begin{proof}[Proof of Theorem \ref{hj}]
Since, System \eqref{bihar2} is linear
and reversible in time, then, by applying the {Lions' {\rm HUM}} \cite{J.L2}, the control problem is reduced to the obtention of suitable observability
inequalities for the adjoint system \eqref{biharr22}, that is, \be C_T^{-1} \|z_0\|_{H^1_0(0,\ell)}^{2}\leq\int_{0}^{T}|\partial_xz(t,0)|^{2}dt\leq C_T \|z_0\|_{H^1_0(0,\ell)}^{2}.\label{biharr25vxsm}\ee
for some positive constant $C_T>0$, depending on $T,$ where $z$ is the solution of the adjoint system \eqref{biharr22}. Let  $\gamma\not\in\Gamma^*,$ where $\Gamma^*$ is given by relation  \eqref{set1}. As a direct application of Proposition \ref{ph1}, both inequalities in \eqref{biharr25vxsm} hold for all $T>0.$  This implies that System \eqref{bihar2} is exactly controllable in time $T>0$ for any initial data $y_0\in H^{-1}(0,\ell)$. Now, let $\gamma\in\Gamma^*$ and let $z$ be the solution of the adjoint system \eqref{biharr22} with initial data $$z_0:=\Phi_{n}(x)=\sqrt{\frac{\ell}{2}}{\phi(x,\lambda_n)},~ \hbox{ for some } n\leq n_0=\[\ell{\pi}^{-1}{\sqrt{|\gamma|}}
\]-1,$$
where the eigenfunctions $\phi(x,\lambda_n)$ are given by \eqref{gdv}. By Theorem \ref{Lem2}, the eigenvalues $\(\lambda_n\)_{n\leq n_0}$ of Problem \eqref{biharr1} are of multiplicity two, and this implies that, $\partial_\lambda\phi(x,\lambda)_{|_{\lambda_n}}=0,$  for each fixed $x\in\(0,\ell\)$. Hence, by relations \eqref{bihar9}, \eqref{7} and the orthogonality properties of the eigenfunctions, we deduce that the solution $z$ of the adjoint system \eqref{biharr22} has the form $$
{z}(t,x)=\sqrt{\frac{\ell}{2}}{e^{i\la_nt}}\(\bar\xi(\lambda_n)\sin(x\xi(\lambda_n))-\xi(\lambda_n)\sin(x\bar\xi(\lambda_n))\), \hbox{ for some } n\leq n_0,$$
where $\xi(\lambda)$ and $\bar\xi(\lambda)$ are given by \eqref{drog}. Therefore, $\partial_x{z}(t,0)=0$. Consequently, the right hand side of the first inequality in \eqref{biharr25vxsm} is zero, while, the left hand side
is not zero. Thus, the  first inequality in \eqref{biharr25vxsm} cannot be valid. The proof is then complete.
\end{proof}


\end{document}